\documentclass[11pt]{amsart}

\usepackage{amsthm}
\usepackage{amsmath}
\usepackage{amssymb}
\usepackage{color}

\newtheorem{thm}{Theorem}[section]
\newtheorem{theorem}[thm]{Theorem}

\newtheorem{corollary}[thm]{Corollary}

\newtheorem{proposition}[thm]{Proposition}

\theoremstyle{remark}

\newtheorem{example}[thm]{Example}

\begin{document}

\title{A note on scalable frames}
\author[Cahill, Chen
 ]{Jameson Cahill and Xuemei Chen}
\address{Department of Mathematics, University
of Missouri, Columbia, MO 65211-4100}

\thanks{The first author was supported by
 NSF DMS 1008183; and  NSF ATD 1042701;  AFOSR  DGE51:  FA9550-11-1-0245}

\email{jameson.cahill@gmail.com}
\email{xuemeic@math.umd.edu}

\begin{abstract}
We study the problem of determining whether a given frame is scalable, and when it is, understanding the set of all possible scalings.  We show that for most frames this is a relatively simple task in that the frame is either not scalable or is scalable in a unique way, and to find this scaling we just have to solve a linear system.  We also provide some insight into the set of all scalings when there is not a unique scaling.  In particular, we show that this set is a convex polytope whose vertices correspond to minimal scalings.
\end{abstract}

\maketitle

\section{Introduction}
A collection of vectors $\{\varphi_i\}_{i=1}^n\subseteq\mathbb{C}^d$ is called a \textit{frame} if there are positive numbers $A\leq B<\infty$ such that
$$
A\|x\|^2\leq\sum_{i=1}^n|\langle x,\varphi_i\rangle|^2\leq B\|x\|^2
$$
for every $x$ in $\mathbb{C}^d$.  If we have $A=B$ we say the frame is \textit{tight}, and if $A=B=1$ we say it is a \textit{Parseval frame}.  Given a frame $\{\varphi_i\}_{i=1}^n$ we define the \textit{frame operator} $S:\mathbb{C}^d\rightarrow\mathbb{C}^d$ by
\begin{equation}\label{FO}
Sx=\sum_{i=1}^n\langle x,\varphi_i\rangle \varphi_i.
\end{equation}
It is easy to see that $S$ is always positive, invertible, and Hermitian.  Furthermore, $\{\varphi_i\}_{i=1}^n$ is a Parseval frame if and only if $S=I_d$ (the identity operator on $\mathbb{C}^d$).  By a slight abuse of notation, given any set of vectors $\{\varphi_i\}_{i=1}^n\subseteq\mathbb{C}^d$ we will refer to the operator defined in \eqref{FO} as their frame operator, even if they do not form a frame (in this case $S$ will not be invertible, but it will still be positive).  If we have that $\|\varphi_i\|=1$ for every $i=1,...,n$ we say it is a \textit{unit norm} frame.  For more background on finite frames we refer to the book \cite{book}.

A frame $\{\varphi_i\}_{i=1}^n$ is said to be \textit{scalable} if there exists a collection of scalars $\{v_i\}_{i=1}^n\subseteq\mathbb{C}$ so that $\{v_i\varphi_i\}_{i=1}^n$ is a Parseval frame.  In this case, we call the vector $(|v_1|^2,...,|v_n|^2)\in\mathbb{R}_+^n$ a \textit{scaling} of $\{\varphi_i\}_{i=1}^n$.  Scalable frames have been studied previously in \cite{scalable}.

We will work in the space $\mathbb{H}_{d\times d}$ of all $d\times d$ Hermitian matrices.  Note that this is a \textbf{real} vector space of dimension $d^2$ (it is not a space over the complex numbers since a Hermitian matrix multiplied by a complex scalar is no longer Hermitian).  The inner product on this space is given by $\langle S,T\rangle=\mathrm{Trace}(ST)$ and the norm induced by this inner product is the Froebenius norm, \textit{i.e.}, $\langle S,S\rangle=\|S\|_F^2$.  

In what follows we will always consider frames in the complex space $\mathbb{C}^d$, however all of our results hold in the real space $\mathbb{R}^d$ as well.  The only difference is in this case we must replace the space $\mathbb{H}_{d\times d}$ with its subspace $\mathbb{S}_{d\times d}$ consisting of all $d \times d$ real symmetric matrices, which is a real vector space of dimension $d(d+1)/2$.  Thus, if one replaces $\mathbb{H}_{d\times d}$ with $\mathbb{S}_{d\times d}$ and $d^2$ with $d(d+1)/2$ all of our results will hold for frames $\{\varphi_i\}_{i=1}^n\subseteq\mathbb{R}^d$ and the same proofs will work.

\section{Scaling generic frames}

Consider the mapping from $\mathbb{C}^d$ to $\mathbb{H}_{d\times d}$ given by
$$
x\mapsto xx^*.
$$
Note that $xx^*$ is the rank one projection onto $\mathrm{span}\{x\}$ scaled by $\|x\|^2$.  $xx^*$ is called the \textit{outer product} of $x$ with itself.  Also note that if $x=\lambda y$ for $\lambda\in\mathbb{C}$ then $xx^*=(\lambda y)(\lambda y)^*=|\lambda|^2yy^*$.

Given a frame $\{\varphi_i\}_{i=1}^n$, in this setting we have that the frame operator is given by
$$
S=\sum_{i=1}^n\varphi_i\varphi_i^*,
$$
so $\{\varphi_i\}_{i=1}^n$ is scalable if and only if there exists a collection of \textbf{positive} scalars $\{w_i\}_{i=1}^n$ so that
$$
\sum_{i=1}^nw_i\varphi_i\varphi_i^*=I_d,
$$
in this case $\{\sqrt{w_i}\varphi_i\}_{i=1}^n$ is a Parseval frame, and the vector $(w_1,...,w_n)\in\mathbb{R}^n_+$ is the scaling.

Before stating our first theorem we need one more definition.  A subset $Q\subseteq\mathbb{R}^n$ is called \textit{generic} if there exists a polynomial $p(x_1,...,x_n)$ such that $Q^c=\{(x_1,...,x_n)\in\mathbb{R}^n:p(x_1,...,x_n)=0\}$.  It is a standard fact that generic sets are open, dense, and full measure.  When we talk about a generic set in $\mathbb{C}^d$ we mean that it is generic when we identify $\mathbb{C}^d$ with $\mathbb{R}^{2d}$.

\begin{theorem}\label{generic}
For a generic choice of vectors $\{\varphi_i\}_{i=1}^{d^2}\subseteq\mathbb{C}^d$ we have that $\mathrm{span}\{\varphi_i\varphi_i^*\}_{i=1}^{d^2}=\mathbb{H}_{d\times d}$.
\end{theorem}
\begin{proof}
First let $\{T_i\}_{i=1}^{d^2}$ be any basis for $\mathbb{H}_{d\times d}$.  Since each $T_i$ is Hermitian we can use the spectral theorem to get a decomposition $T_i=\sum_{j=1}^n\lambda_{ij}P_{ij}$ where each $P_{ij}$ is rank 1.  So it follows that $\mathrm{span}\{P_{ij}\}=\mathbb{H}_{d\times d}$ and therefore this set contains a basis of $\mathbb{H}_{d\times d}.$ Thus, we have constructed a basis of $\mathbb{H}_{d\times d}$ consisting only of rank 1 matrices.

Now observe that for a given choice of vectors $\{\varphi_i\}_{i=1}^{d^2}$ we have that $\mathrm{span}\{\varphi_i\varphi_i^*\}=\mathbb{H}_{d\times d}$ if and only if the determinant of the frame operator is nonzero (note that we are refering to the frame operator of $\{\varphi_i\varphi_i^*\}_{i=1}^{d^2}$ as an operator on $\mathbb{H}_{d\times d}$, not the frame operator of $\{\varphi_i\}_{i=1}^n$ as an operator on $\mathbb{C}^d$).  But the determinant of the frame operator is a polynomial in the (real and imaginary parts) of the entries of the $\varphi_i$'s, and by the first paragraph we know that there is at least one choice for which this does not vanish, so we can conclude that for a generic choice it does not vanish.
\end{proof}

\begin{corollary}
If $n\leq d^2$ then for a generic choice of vectors $\{\varphi_i\}_{i=1}^n\subseteq\mathbb{C}^d$ we have that $\{\varphi_i\varphi_i^*\}_{i=1}^n$ is linearly independent.
\end{corollary}

Given a frame $\{\varphi_i\}_{i=1}^n\subseteq\mathbb{C}^d$ define the operator $\mathcal{A}:\mathbb{R}^n\rightarrow\mathbb{H}_{d\times d}$ by
$$
\mathcal{A}w=\sum_{i=1}^nw_i\varphi_i\varphi_i^*
$$
where $w=(w_1,...,w_n)^T$.  To determine whether $\{\varphi_i\}_{i=1}^n$ is scalable boils down to finding a nonnegative solution to
$$
\mathcal{A}w=I_d.
$$

In the generic case when $\{\varphi_i\varphi_i^*\}_{i=1}^n$ is linearly independent, this system is guaranteed to have either no solution, or one unique solution.  So if it either has a solution with a negative entry or has no solution we can conclude that this frame is not scalable, and if it has a nonnegative solution then it is scalable and this solution tells us the unique scalars to use.  We summarize this in the following corollary:

\begin{corollary}
Given frame $\{\varphi_i\}_{i=1}^n\subseteq\mathbb{C}^d$ such that $\{\varphi_i\varphi_i^*\}_{i=1}^n$ is linearly independent in $\mathbb{H}_{d\times d}$, we can determine its scalability by solving the linear system
\begin{equation}\label{system}
\mathcal{A}w=I_d.
\end{equation}
Furthermore, in this case if it is scalable then it is scalable in a unique way.

In particular, if $n\leq d^2$ then with probability 1, determining the scalability of $\{\varphi_i\}_{i=1}^n$ is equivalent to solving the linear system given in \eqref{system}.
\end{corollary}

\section{Linearly dependent outer products}

In this section we will address the situation when $\{\varphi_i\varphi_i^*\}_{i=1}^n$ is linearly dependent.  The main problem here is that the system $\mathcal{A}w=I_d$ may have many solutions, and possibly none of them are nonnegative.  In this section we will find it convenient to assume that $\|\varphi_i\|=1$ for every $i=1,...,n$, note that we lose no generality by making this assumption.

Given a collection of vectors $\{x_i\}_{i=1}^n\subseteq\mathbb{R}^d$ we define their \textit{affine span} as
$$
\mathrm{aff}\{x_i\}_{i=1}^n:=\{\sum_{i=1}^nc_ix_i:\sum_{i=1}^nc_i=1\}
$$
and we say that $\{x_i\}_{i=1}^n$ is \textit{affinely independent} if
$$
x_j\not\in\mathrm{aff}\{x_i\}_{i\neq j}
$$
for every $j=1,...,n$.  We also define their \textit{convex hull} as
$$
\mathrm{conv}\{x_i\}_{i=1}^n:=\{\sum_{i=1}^nc_ix_i:c_i\geq 0,\sum_{i=1}^nc_i=1\}.
$$
We say a set $\mathcal{P}\subseteq\mathbb{R}^d$ is called a \textit{polytope} if it is the convex hull of finitely many points.

\begin{proposition}
Given a collection of unit norm vectors $\{\varphi_i\}_{i=1}^n\subseteq\mathbb{C}^d$ we have that $\{\varphi_i\varphi_i^*\}_{i=1}^n$ is linearly independent if and only if it is affinely independent.
\end{proposition}
\begin{proof}
Clearly linear independence always implies affine independence.  So suppose that $\{\varphi_i\varphi_i^*\}_{i=1}^n$ is not linearly independent.  Then we have an equation of the form
$$
\varphi_j\varphi_j^*=\sum_{i\neq j}c_i\varphi_i\varphi_i^*
$$
for some $j$.  Also note that since $\|\varphi_i\|=1$ it follows that $\langle\varphi_i\varphi_i^*,I_d \rangle=1$ for every $i=1,...,n$.  Therefore, we have
\begin{eqnarray*}
1&=&\langle\varphi_j\varphi_j^*,I_d \rangle =\langle\sum_{i\neq j}c_i\varphi_i\varphi_i^*,I_d \rangle \\
&=&\sum_{i\neq j}c_i\langle\varphi_i\varphi_i^*,I_d \rangle =\sum_{i\neq j}c_i.
\end{eqnarray*}
Therefore $\{\varphi_i\varphi_i^*\}_{i=1}^n$ is not affinity independent.
\end{proof}

\begin{proposition}\label{b}
A unit norm frame $\{\varphi_i\}_{i=1}^n\subseteq\mathbb{C}^d$ is scalable if and only if $\frac{1}{d}I_d\in\mathrm{conv}\{\varphi_i\varphi_i^*\}_{i=1}^n$.   Furthermore, if $\lambda I_d\in\mathrm{conv}\{\varphi_i\varphi_i^*\}_{i=1}^n$ then $\lambda=\frac{1}{d}$ and if $\sum_{i=1}^nw_i\varphi_i\varphi_i^*=\frac{1}{d}I_d$ then $\sum_{i=1}^nw_i=1$.
\end{proposition}
\begin{proof}
Suppose we have a scaling $w$ so that
$$
I_d=\sum_{i=1}^nw_i\varphi_i\varphi_i^*.
$$
Then
\begin{eqnarray*}
d&=&\langle I_d,I_d \rangle=\langle\sum_{i=1}^nw_i\varphi_i\varphi_i^*,I_d \rangle \\
&=&\sum_{i=1}^nw_i\langle\varphi_i\varphi_i^*,I_d \rangle=\sum_{i=1}^nw_i.
\end{eqnarray*}
Thus, $\sum_{i=1}^n\frac{w_i}{d}=1$ and since $w_i\geq 0$ for every $i=1,...,n$ it follows that $\frac{1}{d}I_d=\sum_{i=1}^n\frac{w_i}{d}\varphi_i\varphi_i^*\in\mathrm{conv}\{\varphi_i\varphi_i^*\}_{i=1}^n$.  The converse is obvious.

The furthermore part follows from a similar argument.  Suppose $\lambda I_d=\sum_{i=1}^nw_i\varphi_i\varphi_i^*$ with $\sum_{i=1}^nw_i=1$.  Then
$$
d\lambda=\langle\lambda I_d,I_d\rangle=\sum_{i=1}^nw_i=1.
$$
Now suppose $\frac{1}{d}I_d=\sum_{i=1}^nw_i\varphi_i\varphi_i^*$.  Then
$$
1=\langle\sum_{i=1}^nw_i\varphi_i\varphi_i^*,I_d \rangle=\sum_{i=1}^nw_i.
$$
\end{proof}

The following theorem is known as Carath\'{e}odory's theorem:

\begin{theorem}\label{thm_car}
Given a set of points $\{x_i\}_{i=1}^n\subseteq\mathbb{R}^d$ suppose $y\in\mathrm{conv}\{x_i\}_{i=1}^n$.  Then there exists a subset $I\subseteq\{1,...,n\}$ such that $y\in\mathrm{conv}\{x_i\}_{i\in I}$ and $\{x_i\}_{i\in I}$ is affinely independent.
\end{theorem}

\begin{corollary}\label{a}
Suppose $\{\varphi_i\}_{i=1}^n\subseteq\mathbb{C}^d$ is a scalable frame.  Then there is a subset $\{\varphi_i\}_{i\in I}$ which is also scalable and $\{\varphi_i\varphi_i^*\}_{i\in I}$ is linearly independent.
\end{corollary}

Given a unit norm frame $\{\varphi_i\}_{i=1}^n\subseteq\mathbb{C}^d$ we define the set
$$
\mathcal{P}(\{\varphi_i\}_{i=1}^n):=\{(w_1,...,w_n):w_i\geq 0,\sum_{i=1}^nw_i\varphi_i\varphi_i^*=\frac{1}{d}I_d\}.
$$
Proposition \ref{b} tells us two things about this set: first we have that $w\in\mathcal{P}(\{\varphi_i\}_{i=1}^n)$ if and only if $d\cdot w$ is a scaling of $\{\varphi_i\}_{i=1}^n$, and second, that $\mathcal{P}(\{\varphi_i\}_{i=1}^n)$ is a (possibly empty) polytope (see, for example, Theorem 1.1 in \cite{polytopes}).

Suppose $\{\varphi_i\}_{i=1}^n\subseteq\mathbb{C}^d$ is a scalable frame, and we are given a scaling $w=(w_1,...,w_n)$.  We say the scaling is \textit{minimal} if $\{\varphi_i:w_i>0\}$ has no proper subset which is scalable.

\begin{theorem}\label{main}
Suppose $\{\varphi_i\}_{i=1}^n\subseteq\mathbb{C}^d$ is a scalable, unit norm frame.  If $w=(w_1,...,w_n)$ is a minmal scaling then $\{\varphi_i\varphi_i^*:w_i>0\}$ is linearly independent.  Furthermore, $\mathcal{P}(\{\varphi_i\}_{i=1}^n)$ is the convex hull of the minimal scalings, \textit{i.e.}, every scaling is a convex combination of minimal scalings.
\end{theorem}
\begin{proof}
The first statement follows directly from Corollary \ref{a}. 

We now show that every vertex of $\mathcal{P}(\{\varphi_i\}_{i=1}^n)$ is indeed a minimal scaling.  Let $u\in\mathcal{P}(\{\varphi_i\}_{i=1}^n)$ be a vertex and assume to the contrary that $u$ is not minimal,   then there exists a $v\in P$ such that $\mathrm{supp}(v)\subsetneq\mathrm{supp}(u)$. Let $w(t)=v+t(u-v)$, and $t_0=\mathrm{min}\{\frac{v_i}{v_i-u_i}:v_i>u_i\}$.  We observe that  $t_0>1$ and $w(t_0)_i\geq0$ since $\rm{supp}(v)\subsetneq\rm{supp}(u)$. This means $w(t_0)\in P$, and $u$ lies on the line segment connecting $v$ and $w(t_0)$ which contradicts the fact that $u$ is a vertex.

Finally we show that every minimal scaling is a vertex of $\mathcal{P}(\{\varphi_i\}_{i=1}^n)$.  Suppose we are given a minimal scaling $w$ which is not a vertex of $\mathcal{P}(\{\varphi_i\}_{i=1}^n)$.  Then we can write $w$ as a convex combination of vertices, say $w=\sum t_iv_i$, where we know at least two $t_i$'s are nonzero, without loss of generality say $t_1$ and $t_2$.  Since both $t_1$ and $t_2$ are positive and all the entries of $v_1$ and $v_2$ are nonnegative, it follows that $\mathrm{supp}(v_1)\cup\mathrm{supp}(v_2)\subseteq\mathrm{supp}(w)$, which contradicts the fact the $w$ is a minimal scaling.
\end{proof}

Theorem \ref{main} reduces the problem of understanding the scalings of the frame $\{\varphi_i\}_{i=1}^n$ to that of finding the vertices of the polytope $\mathcal{P}(\{\varphi_i\}_{i=1}^n)$.  Relatvely fast algorithms for doing this are known, see \cite{algorithm}.

\section{When are outer products linearly independent?}

Since most of the results in this paper deal with linear independence of the outer products of subsets of our frame vectors we will address this issue in this section.  It would be nice if there were conditions on a frame  $\{\varphi_i\}_{i=1}^n$ which could guarantee that the set of outer products $\{\varphi_i\varphi_i^*\}_{i=1}^n$ is linearly independent, or conversely if knowing that $\{\varphi_i\varphi_i^*\}_{i=1}^n$ is linearly independent tells anything about the frame $\{\varphi_i\}_{i=1}^n$.  One obvious condition is that in order for $\{\varphi_i\varphi_i^*\}_{i=1}^n$ to be linearly independent we must have $n\leq d^2$, and when this is satisfied Theorem \ref{generic} tells us that this will usually be the case.

Another condition which is easy to prove is that if $\{\varphi_i\}_{i=1}^n$ is linearly independent then so is $\{\varphi_i\varphi_i^*\}_{i=1}^n$.  The converse of this is certainly not true, and since we are usually interested in frames for which $n>d$ this condition is not very useful.  The main idea here is that while the frame vectors live in a $d$-dimensional space the outer products live in a $d^2$-dimensional space, so there is much more ``room" for them to be linearly independent.

Given a frame $\{\varphi_i\}_{i=1}^n$ we define its \textit{spark} to be the size of its smallest linearly dependent subset, more precisely
$$
\mathrm{spark}(\{\varphi_i\}_{i=1}^n):=\mathrm{min}\{|I|:\{\varphi_i\}_{i\in I}\text{ is linearly dependent}\}.
$$
Clearly for a frame $\{\varphi_i\}_{i=1}^n\subseteq\mathbb{C}^d$ we must have that $\mathrm{spark}(\{\varphi_i\}_{i=1}^n)\leq d+1$, if its spark is equal to $d+1$ we say it is \textit{full spark}.  For more background on full spark frames see \cite{fullspark}.

\begin{proposition}\label{spark}
Suppose $\{\varphi_i\}_{i=1}^n\subseteq\mathbb{C}^d$ is a frame with $n\leq 2d-1$.  If $\{\varphi_i\}_{i=1}^n$ is full spark then $\{\varphi_i\varphi_i^*\}_{i=1}^n$ is linearly independent.
\end{proposition}
\begin{proof}
Suppose by way of contradiction that $\{\varphi_i\}_{i=1}^n$ is full spark but $\{\varphi_i\varphi_i^*\}_{i=1}^n$ is linearly dependent.  Then we can write an equation of the form
\begin{equation*}\label{t}
\sum_{i\in I}a_i\varphi_i\varphi_i^*=\sum_{j\in J}b_j\varphi_j\varphi_j^*
\end{equation*}
with $a_i>0$ for every $i\in I$, $b_j>0$ for every $j\in J$, and $I\cap J=\emptyset$.  This implies that
\begin{eqnarray*}
\mathrm{span}(\{\varphi_i\}_{i\in I})&=&\mathrm{Im}(\sum_{i\in I}a_i\varphi_i\varphi_i^*) \\
&=&\mathrm{Im}(\sum_{j\in J}b_j\varphi_j\varphi_j^*)=\mathrm{span}(\{\varphi_j\}_{j\in J}).
\end{eqnarray*}
But since $n\leq 2d-1$ we have either $|I|\leq d-1$ or $|J|\leq d-1$, so this contradicts the fact the $\{\varphi_i\}_{i=1}^n$ is full spark.
\end{proof}

We first remark that the converse of Proposition \ref{spark} is not true:

\begin{example}\label{ex1}
Let $\{e_1,e_2,e_3\}$ be an orthonormal basis for $\mathbb{C}^3$ and consider the frame $\{e_1,e_2,e_3,e_1+e_2,e_2+e_3\}$.  Clearly this frame is not full spark and yet it is easy to verify that $\{e_1e_1^*,e_2e_2^*,e_3e_3^*,(e_1+e_2)(e_1+e_2)^*,(e_2+e_3)(e_2+e_3)^*\}$ is linearly independent.
\end{example}

Next we remark that the assumption $n\leq 2d-1$ is necessary:

\begin{example}\label{ex2}
Let $\{e_1,e_2\}$ be an orthonormal basis for $\mathbb{C}^2$ and consider the frame $\{e_1,e_2,e_1+e_2,e_1-e_2\}$.  Clearly this frame is full spark but
$$
e_1e_1^*+e_2e_2^*=I_2=\frac{1}{2}((e_1+e_2)(e_1+e_2)^*+(e_1-e_2)(e_1-e_2)^*).
$$
\end{example}

Finally we remark that with only slight modifications the proof of Propostion \ref{spark} can be used to prove the following more general result:

\begin{proposition}\label{spark2}
If $\mathrm{spark}(\{\varphi_i\}_{i=1}^n)\geq s$ then $\mathrm{spark}(\{\varphi_i\varphi_i^*\}_{i=1}^n)\geq 2s-2$.
\end{proposition}

Unfortunately, the converse of Proposition \ref{spark2} is still not true.  The main problem here is that given any three vectors such that no one of them is a scalar multiple of another, the corresponding outer products will be linearly independent (we leave the proof of this as an exercise).  Therefore it is easy to make examples (such as Example \ref{ex1} above) of frames that have tiny spark, but the corresponding outer products are linearly independent.

We conclude our discussion of spark by remarking that in \cite{fullspark} it is shown that computing the spark of a general frame is NP-hard.  Thus, the small amount of insight we gain from Proposition \ref{spark2} is of little practical use.

Another property worth mentioning in this section is known as the \textit{complement property}.  A frame $\{\varphi_i\}_{i=1}^n\subseteq\mathbb{C}^d$ has the complement property if for every $I\subseteq\{1,...,n\}$ we have either $\mathrm{span}(\{\varphi_i\}_{i\in I})=\mathbb{C}^d$ or $\mathrm{span}(\{\varphi_i\}_{i\in I^c})=\mathbb{C}^d$.  We remark that the complement property is usually discussed for frames in a real vector space, but for our purposes it is fine to discuss it for frames in a complex space.  In \cite{phaseless} the complement property was shown to be necessary and sufficient to do phaseless reconstruction in the real case.

If a frame $\{\varphi_i\}_{i=1}^n\subseteq\mathbb{C}^d$ has the complement property then clearly we must have $n\geq 2d-1$ (if not we could partition the frame into two sets each of size at most $d-1$) and that in this case full spark implies the complement property.  If $n=2d-1$ then the complement property is equivalent to full spark, but for $n>2d-1$ the complement property is (slightly) weaker.  One might ask if the complement property tells us anything about the linear independence of the outer products, or vice versa.  Example \ref{ex1} above is an example of a frame which does not have the complement property but the outer products are linearly independent, and Example \ref{ex2} is an example of a frame that does have the complement property but the outer products are linearly dependent.  So it seems like the complement property has nothing to do with the linear independence of the outer products.

Given a frame with the complement property we can add any set of vectors to it without losing the complement property.  Thus it seems natural to ask whether every frame with the complement property has a subset of size $2d-1$ which is full spark.  This also turns out to be not true as the following example shows:

\begin{example}
Consider the frame in Example \ref{ex1} with the vector $e_1+e_3$ added to it.  It is not difficult to verify that this frame does have the complement property, but no subset of size 5 is full spark.
\end{example}

We conclude by noting that as in the proof of Proposition \ref{spark}, a set of outer products $\{\varphi_i\varphi_i^*\}_{i=1}^n$ is linearly dependent if and only if we have an equation of the form
$$
\sum_{i\in I}a_i\varphi_i\varphi_i^*=\sum_{j\in J}b_j\varphi_j\varphi_j^*
$$
with $a_i>0$ for every $i\in I$, $b_j>0$ for every $j\in J$, and $I\cap J=\emptyset$.  This is equivalent to $\{\varphi_i\}_{i=1}^n$ having two disjoint subsets, namely $\{\varphi_i\}_{i\in I}$ and $\{\varphi_j\}_{j\in J}$, which can be scaled to have the same frame operator.  Thus, determining whether $\{\varphi_i\varphi_i^*\}_{i=1}^n$ is linearly independent is equivalent to solving a more difficult scaling problem than the one presented in this paper.

\section*{Acknowledgment}

The authors would like to thank Peter Casazza and Dustin Mixon for insightful conversations during the writing of this paper.

\end{document}